\documentclass[a4paper]{amsart}

\usepackage{amsmath,amssymb,amsthm}
\usepackage{mathrsfs}
\usepackage{indentfirst,color}
\usepackage[colorlinks,citecolor=red,linkcolor=blue,urlcolor=cyan]{hyperref}
\usepackage{graphicx}
\usepackage{tikz,tikz-cd}
\usepackage{bm}

\theoremstyle{plain}%
\newtheorem{theorem}{Theorem}[section]%
\newtheorem{proposition}[theorem]{Proposition}%

\theoremstyle{definition}%
\newtheorem{definition}[theorem]{Definition}%
\newtheorem{remark}[theorem]{Remark}%

\newcommand{\RR}{{\mathbb{R}}}%

\makeatletter
\newsavebox{\@brx}
\newcommand{\llangle}[1][]{\savebox{\@brx}{\(\m@th{#1\langle}\)}%
    \mathopen{\copy\@brx\kern-0.5\wd\@brx\usebox{\@brx}}}
\newcommand{\rrangle}[1][]{\savebox{\@brx}{\(\m@th{#1\rangle}\)}%
    \mathclose{\copy\@brx\kern-0.5\wd\@brx\usebox{\@brx}}}
\makeatother

\newcommand{\loc}{\textup{loc}}%
\newcommand{\nd}{\textup{nd}}%
\newcommand{\image}{\textup{i}}%
\renewcommand{\bar}[1]{\overline{#1}}%
\newcommand{\ddbar}{\image\partial\bar{\partial}}%

\renewcommand{\leq}{\leqslant}%
\renewcommand{\geq}{\geqslant}%

\numberwithin{equation}{section}

\begin{document}

\title[A remark on Bogomolov type vanishing theorem]{A remark on Bogomolov type vanishing theorem}

\author[Xiankui Meng]{Xiankui Meng}
\address{Xiankui Meng: School of Mathematical Sciences and Key Laboratory of Mathematics and Information Networks (Ministry of Education), Beijing University of Posts and Telecommunications, Beijing 100876, China.}
\email{mengxiankui@amss.ac.cn}

\author[Chenghao Qing]{Chenghao Qing}
\address{Chenghao Qing: Yau Mathematical Sciences Center, Tsinghua University, Beijing 100084, China.}
\email{qingchenghao@amss.ac.cn}

\author[Xiangyu Zhou]{Xiangyu Zhou}
\address{Xiangyu Zhou: Institute of Mathematics, Academy of Mathematics and Systems Science, Chinese Academy of Sciences, Beijing 100190, China.}
\email{xyzhou@math.ac.cn}

\date{}

\thanks{The first author was partially supported by the National Natural Science Foundation of China (Grant No. 12271057) and by the National Key R\&D Program of China (Grant No. 2021YFA1002600). 
The second author was supported by the China Postdoctoral Science Foundation (Grant No. 2025M773087). 
The third author was supported by National Key R\&D Program of China (Grant No. 2021YFA1003100) and by the National Natural Science Foundation of China (Grant No. 12288201).
}

\begin{abstract}
    
    Several vanishing theorems for pseudo-effective line bundles are presented. 
    All related results in the present note can be derived from already known theorems, in particular from the Steenbrink vanishing theorem and Boucksom's generalization of Bogomolov's vanishing theorem. 
    No originality or priority is claimed for these statements; they are listed here merely as conjugate forms of the Kawamata--Viehweg vanishing theorem, for convenience of reference and comparison with existing literature.
    
\end{abstract}



\maketitle

\section{Introduction}

Vanishing theorems for cohomology groups valued in coherent analytic sheaves are powerful tools in algebraic geometry and several complex variables. 
Various generalizations of the classical Kodaira vanishing theorem are great developed in this direction and used in classification theory of higher-dimensional projective algebraic varieties.

In complex geometry, the Akizuki-Kodaira-Nakano vanishing theorem asserts that if $L$ is a positive line bundle over a compact K\"ahler manifold $X$ of dimension $n$, then
$$H^q(X,\Omega_X^p\otimes L)=0 \quad \text{ for all } p+q\geq n+1.$$
This was generalized to sheaves of logarithmic differential forms by Norimatsu \cite{Nor78} using analytic methods (see also Deligne-Illusie's proof \cite{DI87} by the characteristic $p$ methods). 
More precisely, they proved that $$H^q(X,\Omega_X^p(\log D)\otimes L)=0  \quad \text{ for all } p+q\geq n+1,$$
where $D$ is a simple normal crossing divisor in $X$. 
We recall the following recent logarithmic type vanishing theorem of Huang-Liu-Wan-Yang \cite{HLWY23}.

\begin{theorem}[Huang-Liu-Wan-Yang] \label{Thm:Huang-Liu-Wan-Yang vanishing}
    Let $X$ be a compact K\"ahler manifold of dimension $n$ and $D=\sum_{i=1}^{s}D_i$ a simple normal crossing divisor on $X$.
    Let $L$ be a line bundle and $\Delta=\sum_{i=1}^{s}a_iD_i$ be an $\mathbb{R}$-divisor with $a_i\in[0,1]$ such that $L\otimes\mathcal{O}_X([\Delta])$ is a $k$-positive $\mathbb{R}$-line bundle, 
    where $\mathcal{O}_X([\Delta])$ denotes the $\mathbb{R}$-line bundle associated with the $\mathbb{R}$-divisor $\Delta$. 
    Then for any nef line bundle $N$, we have 
    $$H^q(X,\Omega^p_X(\log D)\otimes L\otimes N)=0 \quad \text{ for any } p+q\geq n+k+1.$$
    In particular, if $L$ is an ample line bundle, then 
    $$H^q(X,\Omega^p_X(\log D)\otimes L)=0 \quad \text{ for any } p+q\geq n+1.$$
\end{theorem}

The following logarithmic type Serre vanishing theorem is a corollary of Steenbrink's vanishing theorem \cite{Ste85}.
When $\mathcal{F}$ is a line bundle, this can also be derived from Theorem~\ref{Thm:Huang-Liu-Wan-Yang vanishing}.
We give the proof in Section~\ref{Section:Serre Vanishing} for the reader's convenience.

\begin{theorem}\label{MainThm 1}
    Let $X$ be a smooth projective variety of dimension $n$ and $H$ an ample divisor on $X$. 
    For any coherent sheaf $\mathcal{F}$ on $X$, there exists an integer $m_0=m_0(\mathcal{F})$ such that for every smooth divisor $A\in|mH|$ with $m\geq m_0$,
    $$H^q(X,\Omega^p_X(\log A) \otimes \mathcal{F})=0 \quad
    \text{ for all } p+q\geq n+1.$$
\end{theorem}

The notion of multiplier ideal sheaf plays a central role in modern higher-dimensional algebraic geometry. 
It appears naturally when considering singularities of metrics on line bundles. 

Another significant generalization of Kodaira vanishing theorem is Kawamata-Viehweg vanishing theorem, which plays a crucial role in the theory of minimal models for higher-dimensional complex algebraic varieties.
Let us recall here a version of the Kawamata-Viehweg vanishing theorem in the context of nef line bundles and multiplier ideal sheaves (see \cite{DemSmall}).
\begin{theorem}[Kawamata-Viehweg, Demailly] \label{Thm:Kawamata-Viehweg vanishing}
    Let $X$ be a projective algebraic manifold of dimension $n$ and let $L$ be a line bundle over $X$ such that
    some positive multiple $mL$ can be written $mL = F+D$ where $F$ is a nef line bundle and $D$ is an effective divisor.
    Then
    $$H^q(X,K_X\otimes L\otimes\mathcal{I}(m^{-1}D))=0 \quad\text{for } q\geq n-\nd(F)+1,$$
    where $\nd(F)$ is the numerical dimension of $F$.
    In particular, if $L$ is a nef line bundle, then
    $$H^q(X,K_X\otimes L)=0 \quad\text{for } q\geq n-\nd(L)+1.$$
\end{theorem}
See \cite{DemSmall} for a proof via Nadel vanishing theorem and an induction argument, \cite{MZ23-a} for a slight improvement, and \cite{MQZ26} for a generalization to higher direct images.
We also have the famous Bogomolov vanishing theorem \cite{Bog78,Bog80}, which can be seen as a conjugate version of Kawamata-Viehweg vanishing theorem.
\begin{theorem}[Bogomolov]
    Let $X$ be a complex projective manifold of dimension $n$ and let $L$ be a holomorphic line bundle over $X$. Then
    $$H^0(X,\Omega^p_X\otimes L^{-1})=0 \quad \text{ for } p<\kappa(L),$$
    where $\kappa(L)$ denotes the Kodaira--Iitaka dimension of $L$. Equivalently, we have
    $$H^n(X,\Omega^p_X\otimes L)=0 \quad \text{ for } p\geq n-\kappa(L)+1. $$
\end{theorem}

We refer to \cite{Bou02,EV92,Gra15,LMNWZ25,Mou98,SS85,Wat23,Wat26,Wu20} for some generalizations of Bogomolov type vanishing theorems.
We recall a recent Kawamata-Viehweg type (resp. Bogomolov type) result in terms of numerical dimension for closed positive currents on compact K\"ahler manifolds obtained by Cao and Guan-Zhou \cite{Cao14,GZ15} (resp. Li-Meng-Ning-Wang-Zhou \cite{LMNWZ25}) for the reader's convenience and for completeness.
\begin{theorem}[Cao, Guan-Zhou, Li-Meng-Ning-Wang-Zhou]\label{Thm:Li-Meng-Ning-Wang-Zhou vanishing}
    Let $X$ be a compact K\"ahler manifold of dimension $n$. Let $(L,h)$ be a pseudo-effective line bundle on $X$ and let $\nd(\image\Theta_{L,h})$ denote the numerical dimension of $\image\Theta_{L,h}$.
    Then
    \begin{itemize}
        \item[(1)] $H^q(X,K_X\otimes L\otimes\mathcal{I}(h))=0 \text{ for } q\geq n-\nd(\image\Theta_{L,h})+1.$
        In particular, if $(L,h)$ is a big line bundle, then 
        $$H^q(X,K_X\otimes L\otimes\mathcal{I}(h))=0 \quad \text{ for } q\geq 1.$$
        \item[(2)] $H^n(X,\Omega^p_X\otimes L\otimes\mathcal{I}(h))=0 \text{ for } p\geq n-\nd(\image\Theta_{L,h})+1.$
        In particular, if $(L,h)$ is a big line bundle, then 
        $$H^n(X,\Omega^p_X\otimes L\otimes\mathcal{I}(h))=0 \quad \text{ for } p\geq 1.$$
    \end{itemize}
    
\end{theorem}
The definition of $\nd(\image\Theta_{L,h})$ is given in \cite{Cao14}. We omit the precise definition here since it is technical and will not be used in this paper.
Note that the case of $(L,h)$ being big in Theorem~\ref{Thm:Li-Meng-Ning-Wang-Zhou vanishing} (2) is given by Watanabe \cite{Wat23}.

The following Bogomolov type theorem can be regarded as an application of Theorem~\ref{MainThm 1}.
Sometimes we omit the factor $2\pi$ in this paper to simplify the expression.
\begin{theorem}\label{MainThm 2}
    Let $X$ be a projective algebraic manifold of dimension $n$ and let $L$ be a line bundle over $X$.
    Assume that $\{\alpha\}$ is a nef class on $X$ and $c_1(L)-\{\alpha\}=\{\theta+\ddbar\psi\}$, where $\theta$ is a smooth real $(1,1)$-form and $\psi$
    is a function on $X$ such that $\theta+\ddbar\psi\geq0$ in the sense of currents.
    Then
    $$H^n(X,\Omega^p_X\otimes L\otimes\mathcal{I}(\psi))=0 \quad\text{for } p\geq n-\nd(\{\alpha\})+1,$$
    where $\nd(\{\alpha\})$ is the numerical dimension of $\{\alpha\}$.
    In particular, if $L$ is a nef line bundle, then
    $$H^n(X,\Omega^p_X\otimes L)=0 \quad\text{for } p\geq n-\nd(L)+1.$$
\end{theorem}
\begin{remark}
    Under the same assumptions as in Theorem~\ref{MainThm 2}, we obtained a Kawamata-Viehweg type vanishing theorem in \cite{MQZ26}. 
    As a special case of \cite[Theorem 1.9]{MQZ26}, we proved that $$H^q(X,K_X\otimes L\otimes\mathcal{I}(\psi))=0 \quad \text{for } q\geq n-\nd(\{\alpha\})+1. $$   
    This statement forms a conjugate counterpart to Theorem~\ref{MainThm 2}.
    
\end{remark}

In fact, Theorem~\ref{MainThm 2} can be strengthened to the following version.
\begin{theorem}\label{MainThm 3}
    Let $X$ be a compact K\"ahler manifold of dimension $n$ and let $L$ be a pseudo-effective line bundle over $X$.
    Assume that $\{\alpha\}$ is a pseudo-effective class on $X$ and $c_1(L)-\{\alpha\}=\{\theta+\ddbar\psi\}$, where $\theta$ is a smooth real $(1,1)$-form and $\psi$
    is a function on $X$ such that $\theta+\ddbar\psi\geq0$ in the sense of currents.
    Then
    $$H^n(X,\Omega^p_X\otimes L\otimes\mathcal{I}(\psi))=0 \quad\text{for } p\geq n-\nd(\{\alpha\})+1,$$
    where $\nd(\{\alpha\})$ is the numerical dimension of $\{\alpha\}$.
    
\end{theorem}

This result is a corollary of the following Boucksom's generalization of Bogomolov's vanishing theorem.

\begin{theorem}[{\cite{Bou02}}]\label{Theorem:Boucksom}
    Let $X$ be a compact K\"ahler manifold and $L$ a pseudo-effective line bundle over $X$. 
    Then
    \begin{equation*}
        H^0(X,\Omega^p_X\otimes L^{-1})=0 \quad \text{for every } p<\nd(L),
    \end{equation*}
    where $\nd(L)$ denotes the numerical dimension of $L$.
\end{theorem}

The remaining part of this paper is organized as follows. 
In Section~\ref{Section:Preliminary}, we briefly recall some preliminary results about singular Hermitian metrics and sheaves of logarithmic differential forms. 
In Section~\ref{Section:Serre Vanishing}, we prove Theorem~\ref{MainThm 1}.
Then in Section~\ref{Section:Bogomolov Vanishing}, we prove Theorem~\ref{MainThm 2} and Theorem~\ref{MainThm 3}.

\section{Preliminaries} \label{Section:Preliminary}

In this section, we recall some basic results used in this note.

\subsection{Singular Hermitian metrics on line bundles}

Let $(X,\omega)$ be a Hermitian manifold of dimension $n$.
A singular Hermitian metric $h$ on a holomorphic line bundle $L\to X$ is simply a Hermitian metric which can be expressed locally as
$e^{-\varphi_U}$ on $U$ such that $\varphi_U$ is $L^1_\loc$, where
$U\subset X$ is a local coordinate chart such that $L|_U\simeq U\times\mathbb{C}.$
It has a well-defined curvature current $\image\Theta_{L,h}:=\ddbar\varphi_U.$

\begin{definition}
    A function $\varphi:X\rightarrow [-\infty,+\infty)$ is
    said to be quasi-plurisubharmonic (quasi-psh for short) if $\varphi$ is locally the sum of a plurisubharmonic function and a smooth function 
    (or equivalently, if $\ddbar\varphi$ is locally bounded from below). 

    If $\varphi$ is a quasi-psh function on $X$, the multiplier ideal sheaf $\mathcal{I}(\varphi)$ is the ideal subsheaf of $\mathcal{O}_X$ defined by
    $$\mathcal{I}(\varphi)_x=\{f\in\mathcal{O}_{X,x}~|~\exists\ U\ni x\  \text{such that}\ \int_U|f|^2e^{-\varphi}d\lambda<+\infty  \},$$
    where $U$ is an open coordinate neighborhood of $x$ and $d\lambda$ is the standard Lebesgue measure in $\mathbb{C}^n.$
    
    It is easy to see that associated to a singular Hermitian metric $h$ on $L$ satisfying
    $\image\Theta_{L,h}\geq\gamma$ 
    for some smooth real $(1,1)$-form $\gamma$ in the sense of currents, 
    there is a well-defined multiplier ideal sheaf $\mathcal{I}(h)$ on $X$ and $\mathcal{I}(h)$ is coherent (\cite{Nad90}).
\end{definition}
The following strong openness property of multiplier ideal sheaves plays an essential role in the proof of Theorem~\ref{MainThm 2}.

\begin{theorem}[{\cite[Theorem 1.1]{GZ15}}] \label{Thm:strong openness}
    Let $\varphi$ be a negative plurisubharmonic function on $\Delta^n\subset \mathbb{C}^n$, and let $\varphi_0\not\equiv-\infty$ be a negative plurisubharmonic function on $\Delta^n$. 
    Then $\mathcal{I}(\varphi)=\bigcup_{\varepsilon>0}\mathcal{I}(\varphi+\varepsilon\varphi_0)$.
\end{theorem}
We need the following restriction formula obtained in \cite[Theorem 1.1]{Xia22} (see also \cite{FM21,MZ23-b}) for the hyperplane argument in the proof of Theorem~\ref{MainThm 2}.
\begin{theorem}[Xia]\label{Thm:Restriction}
    Let $L$ be a line bundle on a projective manifold $X$
    and $h$ a singular metric on $L$ whose local weight $\varphi$ is quasi-psh. 
    Given any base-point free linear system $\Lambda$ on $X$, there is a pluripolar set $\Sigma\subset\Lambda$  such that for all $H\in\Lambda\setminus\Sigma$, $H$ is smooth and the following sequence
    $$0\rightarrow \mathcal{I}(h)\otimes\mathcal{O}(-H)\rightarrow\mathcal{I}(h)\rightarrow\mathcal{I}(h|_H)\rightarrow 0$$
    is exact.
\end{theorem}

We now recall some definitions of positivity of line bundles (see \cite{DemSmall}).
\begin{definition}
    Let $(X,\omega)$ be a compact K\"ahler manifold of dimension $n$ and $\{\alpha\}\in H^{1,1}(X,\mathbb{R})$ a $(1,1)$-class on $X$. Let $L$ be a line bundle on $X$.
    \begin{itemize}
        \item[(1)] $\{\alpha\}$ is said to be nef if for every $\varepsilon>0$, 
        there exists a smooth function $\varphi_\varepsilon$ on $X$ such that $\alpha+\ddbar \varphi_\varepsilon\geq-\varepsilon\omega$.
        One defines the numerical dimension of $\{\alpha\}$ to be
        $$\nd(\{\alpha\})=\max\{k=0,\ldots,n~|~\{\alpha\}^k\neq0 \text{ in } H^{2k}(X,\mathbb{R})\}.$$
        $L$ is said to be nef if $c_1(L)$ is a nef class.
        In this case, we define its numerical
        dimension by $\nd(L)=\nd(c_1(L))$.
        \item[(2)] $\{\alpha\}$ is said to be pseudo-effective if there is a quasi-psh function $\varphi$ such that $\alpha+\ddbar\varphi\geq 0$ in the sense of currents.
        One defines the numerical dimension of $\{\alpha\}$ to be
        $$\nd(\{\alpha\})=\max\{k=0,\ldots,n~|~\langle \{\alpha\}^k \rangle\neq0 \}.$$
        Here $\langle \{\alpha\}^k \rangle:=\lim\limits_{\delta\to0}\{\langle T^p_{\min,\delta\omega} \rangle\}$ 
        where $T^p_{\min,\delta\omega}$ is the positive current with minimal singularity in the class $\{\alpha+\delta\omega\}$ and $\langle T^p_{\min,\delta\omega} \rangle$ is the non-pluripolar product.
        We refer to \cite{BEGZ10} for more precise argument.
        $L$ is said to be pseudo-effective if $c_1(L)$ is a pseudo-effective class. 
        In this case, we define its numerical
        dimension by $\nd(L)=\nd(c_1(L))$, and there is a singular Hermitian metric $h$ on $L$ such that $\image\Theta_{L,h}\geq 0$ in the sense of currents. 
        We say that $(L,h)$ is pseudo-effective.
        \item[(3)] $\{\alpha\}$ is said to be big if there is a quasi-psh function $\varphi$ such that $\alpha+\ddbar\varphi\geq \varepsilon\omega$ for some positive continuous function $\varepsilon$ in the sense of currents.
        $L$ is said to be big if $c_1(L)$ is a big class. 
        In this case, there is a singular Hermitian metric $h$ on $L$ such that $\image\Theta_{L,h}\geq \varepsilon\omega$ in the sense of currents. 
        We say that $(L,h)$ is big.
        Note that a nef line bundle is big if its numerical dimension is $n$.
    \end{itemize}
\end{definition}
We remark that if $\{\alpha\}$ is a nef class on a projective manifold $X$ with $\nd(\{\alpha\})\leq\dim X-1$ and $S\subset X$ is a smooth ample divisor, then $\{\alpha\}|_S$ is also nef and $\nd(\{\alpha\}|_S)=\nd(\{\alpha\})$. 
See \cite[Proposition 6.21]{DemSmall} for a proof.
One can also define the numerical dimension of any class $\{\alpha\}\in H^{1,1}(X,\RR)$ (for example, see \cite[Section 19]{DemSmall}).

\subsection{Sheaf of logarithmic differential forms}
We recall here some basic properties of the sheaf of logarithmic differential forms (see \cite[Chapter 2]{EV92}).
A reduced divisor $D=\sum_{i=1}^{s}D_i$ on a projective manifold $X$ is called a simple normal crossing divisor if every irreducible component $D_j$ is smooth and all intersections are transverse.
The sheaf $\Omega_X^p(\log D)$ is the sheaf of germs of differential $p$-forms on $X$ with at most logarithmic poles along $D$.
Sections of the logarithmic sheaf $\Omega_X^p(\log D)$ on an open subset $U$ are 
$$\Gamma(U,\Omega_X^p(\log D)):=\{\alpha\in\Gamma(U,\Omega_X^p\otimes \mathcal{O}_X(D)) ~|~d\alpha\in\Gamma(U,\Omega_X^{p+1}\otimes \mathcal{O}_X(D)) \}.$$
It is easy to see that when $p=\dim X$, $\Omega_X^p(\log D)=K_X\otimes\mathcal{O}_X(D)$.
\begin{proposition}\label{Prop:Exact sequence} 
    Let $D=\sum_{i=1}^{s}D_i$ be a reduced simple normal crossing divisor on a projective manifold $X$ of dimension $n$.
    Then the following properties hold.
    \begin{enumerate}
        \item The sheaf $\Omega_X^p(\log D)$ is locally free.
        \item One has three exact sequences:
        \begin{itemize}
            \item[(a)] $$0\rightarrow \Omega_X^1\rightarrow\Omega_X^1(\log D)\rightarrow \bigoplus_{i=1}^s\mathcal{O}_{D_i}\rightarrow 0.$$
            \item[(b)] $$0\rightarrow\Omega_X^p(\log(D-D_1))\rightarrow \Omega_X^p(\log D)\rightarrow \Omega_{D_1}^{p-1}(\log(D-D_1)|_{D_1})\rightarrow 0.$$
            \item[(c)] $$0\rightarrow\Omega_X^p(\log D)(-D_1)\rightarrow\Omega_X^p(\log(D-D_1))\rightarrow\Omega^p_{D_1}(\log(D-D_1)|_{D_1})\rightarrow0.$$
        \end{itemize}
        \item Let $L$ be a line bundle on $X$. Logarithmic Serre duality gives an isomorphism
        $$H^q(X,\Omega_X^p(\log D)\otimes L)\cong H^{n-q}(X,\Omega_X^{n-p}(\log D)\otimes L^{-1}\otimes \mathcal{O}_X(D)^{-1})^*.$$
    \end{enumerate}
\end{proposition}

\section{Logarithmic type Serre vanishing theorem} \label{Section:Serre Vanishing}

We now prove logarithmic type Serre vanishing Theorem~\ref{MainThm 1}.
The argument is standard.
\begin{proof}[Proof of Theorem~\ref{MainThm 1}]
    We first prove the statement in the case where the coherent sheaf is a line bundle $L$. 
    
    If $p=n$, by Kodaira vanishing theorem, we know that
    $$H^q(X,\Omega^n_X(\log A) \otimes L)=H^q(X,K_X\otimes \mathcal{O}_X(A)\otimes L)=0 \quad \text{ for all } q\geq 1.$$
    
    Assume now that $p\leq n-1$.
    Taking $D=D_1=A$ in the exact sequence of Proposition~\ref{Prop:Exact sequence} (2.c) and tensoring with $\mathcal{O}_X(A)\otimes L$, the sequence
    $$0\rightarrow\Omega_X^p(\log A)\otimes L\rightarrow\Omega_X^p\otimes\mathcal{O}_X(A)\otimes L\rightarrow\Omega^p_A\otimes(\mathcal{O}_X(A)\otimes L)|_A\rightarrow0$$
    is exact.
    We have the associated long exact sequence
    $$H^{q-1}(A,\Omega^p_A\otimes(\mathcal{O}_X(A)\otimes L)|_A)\rightarrow H^q(X,\Omega_X^p(\log A)\otimes L)\rightarrow H^q(X,\Omega_X^p\otimes\mathcal{O}_X(A)\otimes L).$$
    By Akizuki-Kodaira-Nakano vanishing theorem, the first term and the third term vanish if $p+q\geq n+1$ and $p\leq n-1$.
    Hence the middle term vanishes for $p+q\geq n+1$.
    Therefore, we have proved that for any line bundle $L$ on $X$, there exists an integer $m_0=m_0(L)$ such that for every smooth divisor $A\in|mH|$ with $m\geq m_0$,
    $$H^q(X,\Omega^p_X(\log A) \otimes L)=0 \quad
    \text{ for all } p+q\geq n+1.$$
    
    We now prove the statement for any coherent sheaf by descending induction on the cohomological degree $q$.
    
    For the given coherent sheaf $\mathcal{F}$, choose an integer $k\gg 0$ such that $\mathcal{F}\otimes \mathcal{O}_X(kH)$ is globally generated.
    Then there exists a surjection
    $$\bigoplus_{i=1}^r \mathcal{O}_X(-kH)\twoheadrightarrow \mathcal{F}.$$
    Set $L_i=\mathcal{O}_X(-kH)$ for $i=1,\ldots, r$ and denote the kernel by $\mathcal{K}$, so that we have an exact sequence
    \begin{equation*}
        0\rightarrow\mathcal{K}\rightarrow\bigoplus_{i=1}^rL_i\rightarrow\mathcal{F}\rightarrow 0.
    \end{equation*}
    By the first part of the proof, there exists an integer $m_1$ such that for every $m\geq m_1$ and smooth $A\in|mH|$, 
    \begin{equation}\label{Formula:vanishing for line bundle}
        H^q(X,\Omega^p_X(\log A) \otimes L_i)=0
        \quad \text{ for all } p+q\geq n+1.
    \end{equation}
    Consider the exact sequence
    \begin{equation}\label{short sequence}
        0\rightarrow\mathcal{K}\otimes\Omega_X^p(\log A)\rightarrow\bigoplus_{i=1}^rL_i\otimes\Omega_X^p(\log A)\rightarrow\mathcal{F}\otimes\Omega_X^p(\log A)\rightarrow 0
    \end{equation}
    and the associated long exact sequence
    $$\bigoplus_{i=1}^rH^n(X,\Omega_X^p(\log A)\otimes L_i)\rightarrow H^n(X,\Omega_X^p(\log A)\otimes\mathcal{F})\rightarrow H^{n+1}(X,\Omega_X^p(\log A)\otimes\mathcal{K}).$$
    Note that (\ref{Formula:vanishing for line bundle}) implies that 
    $$\bigoplus_{i=1}^rH^n(X,\Omega_X^p(\log A)\otimes L_i)=0 \quad \text{ for } m\geq m_1 \text{ and } p\geq 1,$$
    and 
    $$H^{n+1}(X,\Omega_X^p(\log A)\otimes\mathcal{K})=0$$
    since $\dim X=n$.
    Thus, 
    $$H^n(X,\Omega_X^p(\log A)\otimes\mathcal{F})=0 \quad \text{ for } m\geq m_1 \text{ and } p\geq 1,$$
    that is, the theorem holds for $q=n$.
    
    Let $q\leq n-1$ and assume that the theorem holds for
    every coherent sheaf and for all cohomological degrees strictly larger than $q$ (i.e., for $q+1,q+2,...,n$).
    Moreover, the statement is valid in the following uniform sense: for any coherent sheaf $\mathcal{G}$, there exists an integer $m(\mathcal{G})$ such that for all $m\geq m(\mathcal{G})$ and smooth $A\in|mH|$, 
    $$H^l(X,\Omega_X^p(\log A)\otimes\mathcal{G})=0 \quad \text{ for } p+l\geq n+1 \text{ with } l\geq q+1.$$
    
    Consider the long exact sequence associated to the exact sequence (\ref{short sequence})
    \begin{align*}
        \cdots\rightarrow& \bigoplus_{i=1}^rH^q(X,\Omega_X^p(\log A)\otimes L_i)\rightarrow H^q(X,\Omega_X^p(\log A)\otimes\mathcal{F}) \\ 
        \rightarrow& H^{q+1}(X,\Omega_X^p(\log A)\otimes\mathcal{K})\rightarrow\bigoplus_{i=1}^rH^{q+1}(X,\Omega_X^p(\log A)\otimes L_i)\rightarrow\cdots
    \end{align*}
    By (\ref{Formula:vanishing for line bundle}), we have 
    $$\bigoplus_{i=1}^rH^q(X,\Omega_X^p(\log A)\otimes L_i)=0 \quad \text{ for } m\geq m_1 \text{ and } p+q\geq n+1.$$
    Therefore, the first term and the fourth term vanish.
    Since $\mathcal{K}$ is a coherent sheaf, the third term vanishes for $m$ large enough by the induction hypothesis applied to $\mathcal{K}$ at degree $q+1$,
    that is,
    $$H^{q+1}(X,\Omega_X^p(\log A)\otimes\mathcal{K})=0 \quad \text{ for } m\geq m(\mathcal{K}) \text{ and } p+q\geq n+1.$$
    This implies that there exists an integer $m_0=m_0(\mathcal{F})$ such that for every smooth divisor $A\in|mH|$ with $m\geq m_0$,
    $$H^q(X,\Omega^p_X(\log A) \otimes \mathcal{F})=0 \quad
    \text{ for all } p+q\geq n+1.$$
    
    This completes the proof of Theorem~\ref{MainThm 1}.
\end{proof}

\begin{remark}
    In general, one cannot expect that 
    $$H^q(X,\Omega_X^p(\log A)\otimes \mathcal{F})=0 \quad \text{ for all } q\geq 1 \text{ and } p\geq 1$$
    for $m$ sufficiently large.
    We provide a counterexample here. 
    Let $X=\mathbb{P}^2$ be the projective space of dimension two and $H=\mathcal{O}_X(1)$ the hyperplane bundle. 
    Consider the coherent sheaf $\mathcal{O}_X$.
    We show that $H^1(X,\Omega_X^1(\log A))\neq 0$ for $A\in |mH|$ with $m\geq 4$.
    
    Consider the Poincar\'e residue sequence
    $$0\rightarrow \Omega_X^1\rightarrow \Omega_X^1(\log A)\rightarrow i_*\mathcal{O}_A\rightarrow 0,$$
    where $i: A\hookrightarrow X$ is the inclusion,
    and the associated long exact sequence 
    \begin{align*}
        \cdots\rightarrow H^1(X,\Omega_X^1)
        \rightarrow H^1(X,\Omega_X^1(\log A))
        \rightarrow H^1(A,\mathcal{O}_A)
        \rightarrow H^2(X,\Omega_X^1)
        \rightarrow\cdots.
    \end{align*}
    Note that $H^2(X,\Omega_X^1)=0$. 
    Thus $\dim H^1(X,\Omega_X^1(\log A))\geq \dim H^1(A,\mathcal{O}_A)$.
    Since $$\dim H^1(A,\mathcal{O}_A)=g=\frac{(m-1)(m-2)}{2}\geq 1,$$
    where $g$ is the genus of the smooth curve $A$.
    This shows that Theorem~\ref{MainThm 1} does not hold without the assumption $p+q\geq n+1$.
\end{remark}

\section{Bogomolov type vanishing theorem} \label{Section:Bogomolov Vanishing}

In this section, we first prove Theorem~\ref{MainThm 2} on projective manifolds.
\begin{proof}[Proof of Theorem~\ref{MainThm 2}]
    Let us first assume that $\nd(\{\alpha\})=n$. In this case, $\{\alpha\}$ is a big class.
    Therefore, there is a quasi-psh function $\varphi$ on $X$ such that $\alpha+\ddbar\varphi\geq\omega$ for some K\"ahler metric $\omega$ on $X$.
    On the other hand, since $\{\alpha\}$ is nef, there exists a smooth function $\varphi_\varepsilon$ such that $\alpha+\ddbar\varphi_\varepsilon\geq-\varepsilon\omega$ for every $\varepsilon>0$.
    Let $h_\infty$ be a smooth metric on $L$.
    Since $c_1(L)-\{\alpha\}=\{\theta+\ddbar\psi\}$ and $\theta+\ddbar\psi\geq 0$, there exists a smooth function $\psi'$ on $X$ such that
    $$\image\Theta_{L,h_\infty}-\alpha+\ddbar\psi'+\ddbar\psi=\theta+\ddbar\psi\geq 0.$$
    Define a singular metric $h_L$ on $L$ by 
    $h_L=h_\infty e^{-\psi'-\psi-(1-\delta)\varphi_\varepsilon-\delta\varphi}$ 
    with $\varepsilon\ll\delta\ll 1.$
    Its curvature current satisfies
    \begin{align*}
        \image\Theta_{L,h_L}&=\image\Theta_{L,h_\infty}+\ddbar\psi'+\ddbar\psi+(1-\delta)\ddbar\varphi_\varepsilon+\delta\ddbar\varphi \\
        &\geq (1-\delta)(\alpha+\ddbar\varphi_\varepsilon)+\delta(\alpha+\ddbar\varphi) \\
        &\geq -(1-\delta)\varepsilon\omega+\delta\omega \\
        &\geq \delta\varepsilon\omega.
    \end{align*}
    Since $\varphi_\varepsilon$ and $\psi'$ are smooth, the strong openness property (Theorem~\ref{Thm:strong openness}) implies that 
    $$\mathcal{I}(h_L)=\mathcal{I}(\psi+\delta\varphi)=\mathcal{I}(\psi)$$
    for $\delta\ll 1$.
    Therefore, by using Theorem~\ref{Thm:Li-Meng-Ning-Wang-Zhou vanishing}, we have 
    $$H^n(X,\Omega^p_X\otimes L\otimes\mathcal{I}(\psi))=H^n(X,\Omega^p_X\otimes L\otimes\mathcal{I}(h_L))=0 \quad \text{ for } p\geq 1.$$
    
    We prove the general case by induction on the dimension $n$ of $X$ and the use of suitable hyperplane sections.
    
    If $\dim X=1$ and $\nd(\{\alpha\})=0$, the statement of Theorem~\ref{MainThm 2} is trivial.
    If $\dim X=1$ and $\nd(\{\alpha\})=1$, then $\nd(\{\alpha\})=\dim X$ and we can apply the first part of our arguments to conclude that the theorem is valid in this special case.
    So Theorem~\ref{MainThm 2} is proved when $\dim X=1$.
    Suppose the conclusion holds on all projective manifolds of dimension $n-1$.
    
    According to Theorem~\ref{Thm:Restriction}, one can take a smooth divisor $A$ ample enough such that the sequence
    $$0\rightarrow\mathcal{I}(\psi)\otimes\mathcal{O}_X(-A)\rightarrow\mathcal{I}(\psi)\rightarrow\mathcal{I}(\psi|_A)\rightarrow 0$$
    is exact.
    By Proposition \ref{Prop:Exact sequence} (2.b), we have the Poincar\'e residue sequence
    $$0\rightarrow \Omega_X^p\rightarrow \Omega_X^p(\log A)\rightarrow i_*\Omega_A^{p-1}\rightarrow 0,$$
    where $i: A\hookrightarrow X$ is the inclusion.
    Consider the exact sequence 
    $$0\rightarrow \Omega_X^p\otimes L\otimes\mathcal{I}(\psi)\rightarrow \Omega_X^p(\log A)\otimes L\otimes\mathcal{I}(\psi)\rightarrow i_*\Omega_A^{p-1}\otimes L|_A\otimes\mathcal{I}(\psi|_A)\rightarrow 0.$$
    We have the associated long exact sequence 
    \begin{align*}
        \cdots&\rightarrow H^{n-1}(A,\Omega_A^{p-1}\otimes L|_A\otimes\mathcal{I}(\psi|_A))
        \rightarrow H^n(X,\Omega_X^p\otimes L\otimes\mathcal{I}(\psi)) \\
        &\rightarrow H^n(X,\Omega_X^p(\log A)\otimes L\otimes\mathcal{I}(\psi)) .
    \end{align*}
    By taking $A$ ample enough, we have 
    $$H^n(X,\Omega_X^p(\log A)\otimes L\otimes\mathcal{I}(\psi))=0 \quad \text{ for } p\geq 1$$
    by Theorem~\ref{MainThm 1}.
    On the other hand, the induction hypothesis implies that
    $$H^{n-1}(A,\Omega_A^{p-1}\otimes L|_A\otimes\mathcal{I}(\psi|_A))=0 \quad \text{ for } p-1\geq n-1-\nd(\{\alpha\}|_A)+1.$$
    Since the numerical dimension $\nd(\{\alpha\})$ is invariant under slicing if $\nd(\{\alpha\})\leq n-1$, we have $\nd(\{\alpha\}|_A)=\nd(\{\alpha\})$. 
    Thus 
    $$H^n(X,\Omega_X^p\otimes L\otimes\mathcal{I}(\psi))=0 \quad \text{ for } p\geq n-\nd(\{\alpha\})+1.$$
    
    This completes the proof of Theorem~\ref{MainThm 2}.
\end{proof}

We now prove Theorem~\ref{MainThm 3}.

\begin{proof}[Proof of Theorem~\ref{MainThm 3}]
    Note that $\mathcal{I}(\psi)$ is a torsion-free ideal subsheaf in $\mathcal{O}_X$.
    Then $\mathcal{I}(\psi)^{\vee\vee}$ is a rank-one reflexive sheaf. 
    Thus $M:=\mathcal{I}(\psi)^{\vee\vee}$ is a holomorphic line bundle.
    We have an inclusion $\mathcal{I}(\psi)^{\vee\vee}\hookrightarrow\mathcal{O}_X$ induced by the inclusion $\mathcal{I}(\psi)\hookrightarrow\mathcal{O}_X$.
    This inclusion gives a nonzero section $s\in H^0(X,M^{-1})$.
    Let $D=\sum_{i}a_iD_i$ be the zero divisor of section $s$. 
    We have $M^{-1}\simeq\mathcal{O}_X(D)$ and $M\simeq\mathcal{O}_X(-D)$.
    
    On the other hand, we have the following expression of $\mathcal{I}(\psi)$ by \cite[Proposition~3.2]{DP03}. 
    Consider the Siu decomposition of the closed positive $(1,1)$-current $T:=\theta+\ddbar\psi$:
    $$T=\sum_{j=1}^{\infty}\lambda_j[E_j]+R,$$
    where $E_j$ are effective divisors and $R$ is the residue closed positive current such that the Lelong sublevel sets $E_c(R)$, $c>0$, all have codimension two.
    Then we have 
    $$\mathcal{I}(\psi)\subset\mathcal{O}_X(-\sum\lfloor \lambda_j \rfloor E_j),$$
    and the equality holds on $X\setminus Z$ where $Z$ is an analytic subset of $X$ whose components all have codimension at least two.
    
    Therefore, we have $M=\mathcal{O}_X(-\sum\lfloor \lambda_j \rfloor E_j)$. 
    Consider the pseudo-effective line bundle $N:=L\otimes M$. 
    Since $$c_1(N)-\{\alpha\}=c_1(L)-\{\alpha\}+c_1(M)=\{T-\sum\lfloor \lambda_j \rfloor E_j\}$$
    is pseudo-effective, 
    we have $\nd(c_1(N))\geq\nd(c_1(\alpha))$.
    By Serre duality and (1), 
    \begin{align*}
        H^n(X,\Omega_X^p\otimes L\otimes\mathcal{I}(\psi))^*&\cong
        H^0(X,\Omega_X^{n-p}\otimes L^{-1}\otimes\mathcal{I}(\psi)^\vee) \\
        &\cong H^0(X,\Omega^{n-p}_X\otimes N^{-1}) \\
        &=0
    \end{align*}
    for $n-p<\nd(N)$, equivalently, $p\geq n-\nd(N)+1$.
    Therefore, 
    $$H^n(X,\Omega_X^p\otimes L\otimes\mathcal{I}(\psi))=0 \quad \text{for } p\geq n-\nd(\alpha)+1.$$
    
    This completes the proof of Theorem~\ref{MainThm 3}.
\end{proof}

\end{document}